\documentclass[leqno,11pt]{amsart}

\usepackage{amssymb, amsmath}
\usepackage{amssymb,amsfonts}
\usepackage{amsmath,latexsym}
\usepackage{pdfsync}
\usepackage{bbm}
\usepackage{xcolor}

\let\al=\alpha
\let\b=\beta

\let\d=\delta
\let\e=\varepsilon

\let\lam=\lambda

\let\D=\Delta

\let\wh=\widehat

\def\cF{{\mathcal F}}

\newtheorem{theo}{Theorem}[section]

\newtheorem{lemme}{Lemma}[section]

\newtheorem{cor}{Corollary}[section]

\numberwithin{equation}{section}

\def\sumetage#1#2{\sum_{\substack{{#1}\\{#2}}}}

\makeatletter
\def\sommaire{\@restonecolfalse\if@twocolumn\@restonecoltrue\onecolumn
\fi\chapter*{Sommaire\@mkboth{SOMMAIRE}{SOMMAIRE}}
  \@starttoc{toc}\if@restonecol\twocolumn\fi}
\makeatother

\makeatletter
\def\thebibliographie#1{\chapter*{Bibliographie\@mkboth
  {BIBLIOGRAPHIE}{BIBLIOGRAPHIE}}\list
  {[\arabic{enumi}]}{\settowidth\labelwidth{[#1]}\leftmargin\labelwidth
  \advance\leftmargin\labelsep
  \usecounter{enumi}}
  \def\newblock{\hskip .11em plus .33em minus .07em}
  \sloppy\clubpenalty4000\widowpenalty4000
  \sfcode`\.=1000\relax}

\makeatother

\makeatletter
\def\references#1{\section*{R\'ef\'erences\@mkboth
  {R\'EF\'ERENCES}{R\'EF\'ERENCES}}\list
  {[\arabic{enumi}]}{\settowidth\labelwidth{[#1]}\leftmargin\labelwidth
  \advance\leftmargin\labelsep
  \usecounter{enumi}}
  \def\newblock{\hskip .11em plus .33em minus .07em}
  \sloppy\clubpenalty4000\widowpenalty4000
  \sfcode`\.=1000\relax}

\makeatother

\def\refer#1{~\ref{#1}}
\def\refeq#1{~(\ref{#1})}

\def\longformule#1#2{
\displaylines{
\qquad{#1}
\hfill\cr
\hfill {#2}
\qquad\cr
}
}
\def\inte#1{
\displaystyle\mathop{#1\kern0pt}^\circ
}
\def\sumetage#1#2{\sum_{\substack{{#1}\\{#2}}}}

\newcommand{\beq}{\begin{equation}}
\newcommand{\eeq}{\end{equation}}
\newcommand{\ben}{\begin{eqnarray}}
\newcommand{\een}{\end{eqnarray}}
\newcommand{\beno}{\begin{eqnarray*}}
\newcommand{\eeno}{\end{eqnarray*}}

\let\al=\alpha
\let\b=\beta

\let\d=\delta
\let\e=\varepsilon

\let\lam=\lambda

\let\D=\Delta

\let\wh=\widehat

\def\cF{{\mathcal F}}

\def\virgp{\raise 2pt\hbox{,}}
\def\cdotpv{\raise 2pt\hbox{;}}

\def\eqdefa{\buildrel\hbox{\footnotesize def}\over =}

\def\C{\mathop{\mathbb C\kern 0pt}\nolimits}
\def\DD{\mathop{\mathbb D\kern 0pt}\nolimits}
\def\EE{\mathop{\mathbb E\kern 0pt}\nolimits}
\def\K{\mathop{\mathbb K\kern 0pt}\nolimits}
\def\N{\mathop{\mathbb  N\kern 0pt}\nolimits}
\def\Q{\mathop{\mathbb  Q\kern 0pt}\nolimits}
\def\R{{\mathop{\mathbb R\kern 0pt}\nolimits}}
\def\SS{\mathop{\mathbb  S\kern 0pt}\nolimits}
\def\St{\mathop{\mathbb  S\kern 0pt}\nolimits}
\def\Z{\mathop{\mathbb  Z\kern 0pt}\nolimits}
\def\ZZ{{\mathop{\mathbb  Z\kern 0pt}\nolimits}}
\def\H{{\mathop{{\mathbb  H\kern 0pt}}\nolimits}}
\def\PP{\mathop{\mathbb P\kern 0pt}\nolimits}
\def\TT{\mathop{\mathbb T\kern 0pt}\nolimits}

\def\pa{\partial}

\newcommand{\ds}{\displaystyle}

\newcommand{\with}{\quad\hbox{with}\quad}

\def\dive{\mathop{\rm div}\nolimits}

\begin{document}

  \title[On the radius of analyticity of solutions to semi-linear  parabolic systems] {On the radius of analyticity of solutions to semi-linear  parabolic systems}

 \author[J.-Y. Chemin]{Jean-Yves  Chemin}
\address[J.-Y. Chemin]%
{Laboratoire Jacques Louis Lions - UMR 7598,  Sorbonne Universit\'e\\ Bo\^\i te courrier 187, 4 place Jussieu, 75252 Paris
Cedex 05, France}
\email{chemin@ann.jussieu.fr }
\author[I. Gallagher]{Isabelle Gallagher}
\address[I. Gallagher]%
{DMA, \'Ecole normale sup\'erieure, CNRS, PSL Research University, 75005 Paris
 \\
and UFR de math\'ematiques, Universit\'e Paris-Diderot, Sorbonne Paris-Cit\'e, 75013 Paris, France.}
\email{gallagher@math.ens.fr}
\author[P. ZHANG]{Ping Zhang}%
\address[P. Zhang]
 {Academy of
Mathematics $\&$ Systems Science and  Hua Loo-Keng Key Laboratory of
Mathematics, The Chinese Academy of Sciences, Beijing 100190, China, and School of Mathematical Sciences,
University of Chinese Academy of Sciences, Beijing 100049, China.}
\email{zp@amss.ac.cn}
\subjclass[2010]{}
\begin{abstract}
We study the radius of analyticity~$R(t)$ in space, of
strong solutions to systems of scale-invariant semi-linear parabolic equations. It is well-known that near the initial time,~$R(t)t^{-\frac12}$ is bounded from below by a positive constant. In this paper we   prove
 that~$\displaystyle\liminf_{t\rightarrow 0} R(t)t^{-\frac12}= \infty$, and
assuming higher
regularity for the initial data, we obtain an improved lower bound near time zero.
As an application, we prove that for any global solution~$u\in C([0,\infty); H^{\frac12}(\R^3))$ of  the Navier-Stokes
equations, there holds~$\displaystyle\liminf_{t\rightarrow \infty} R(t)t^{-\frac12}= \infty$.
\end{abstract}

\maketitle


\setcounter{equation}{0}
\section{Introduction}
We consider the following system of~$N$ equations on~$\R^+\times \R^d$:
$$
{\rm(SP)} \quad \left\{
\begin{array}{c}
\partial_t U  -\D U  =  P(U)\\
U_{|t=0} =U_0\,,
\end{array}
\right. \with
P_j(U)\eqdefa \sumetage {\ell\in \N^N} {|\ell|=k }  A_{j,\ell} (D)( U^\ell) \quad\hbox{for~$j$ in $\{1,\dots, N\}$}
$$
where~$A_{j,\ell}(D)$ are homogeneous Fourier multipliers  of degree~$\b\in [0,2[,$ and~$U=(U_j)_{1\leq j \leq N}$.  The order of the nonlinearity is~$k\geq 2$ and we have written~$\displaystyle U^\ell = \prod_{j=1}^N U_{j}^{\ell_j}$.
An important property of a such a system is its   scaling invariance: if a function~$U$ satisfies~(SP)  on a time interval~$[0,T]$ with the initial data~$U_0$, then
the function~$U_\lam$ defined by
$$
U_\lam(t,x)\eqdefa \lam^\al U(\lam^2 t,\lam x)
$$
satisfies~(SP) on the time interval~$[0,\lam^{-2} T]$ with the initial data~$U_{0,\lam}\eqdefa \lam^\al U_0(\lam\,\cdot)$ for
$$
\al\eqdefa \frac {2-\b} {k-1}\,\cdotp
$$
Note that~$\alpha$ is positive, and in the following we shall assume that~$ \alpha \leq d/k$.
For example  for the Navier-Stokes equations there holds~$\b=1$ and~$k=2$, while    for the cubic heat equation there holds~$\b=0$ and~$k=3$. In both cases~$\al=1$.
The scaling invariant   Sobolev space for the initial data is~$H^{ s_{\rm crit} } (\R^d)$, with~$ s_{\rm crit}  \eqdefa \frac d 2-\al$, recalling that~$H^{ s } (\R^d)$ is defined by the following norm, for~$s<d/2$:
$$
\|f\|_{H^{ s } (\R^d)}\eqdefa \Big(\int_{\R^d} |\xi|^{2s} |\widehat f(\xi)|^2\, d\xi\Big)^\frac12\,,
$$
where~$\widehat f =\mathcal F f$ is the Fourier transform of~$f$.

The question of solving the Cauchy problem for systems such as~(SP) in scale invariant spaces has been widely studied.
We shall make no attempt at listing all the results on the subject but simply recall the typical so-called Kato-type theorem, which may be proved by a Banach fixed point argument (see for instance~\cite{bcdbookk,giga,kato,ribaud,weissler} among others)
\begin{theo}
\label {KatoAbstrait}
{\sl Assume~$\alpha \in \left]1/k,d/k\right]$ and define~$s_{\rm crit}  \eqdefa   d /2-\al$. Then, for any~$\d\in [0, \al[$,  for any initial data~$U_0  $ belonging to~$ H^{ s_{\rm crit}+\d } (\R^d)$, a positive time~$T$ exists such that  the system~(SP) has a unique solution~$U$ in the Kato space~$K^p_T$  such that
\beq\label{Kato}
\|U\|_{K^p_T} \eqdefa \sup_{t\leq T} t^{\frac 1 p} \|U(t)\|_{H^{s_{\rm crit} +\frac 2 p}} <\infty\,.
\eeq
Moreover, if~$\d$ is positve, a constant~$c$ exists such that~$T\geq c \|U_0\|_{H^{s_{\rm crit} +\d}} ^{ -\frac 2 \d}$.}
\end{theo}
The goal of this article is to analyze the instantaneous smoothing effect of (SP): let us recall  that in~\cite{foiastemam}, the analyticity of smooth periodic solutions  to the Navier-Stokes equations~(NS) is proved, in the sense that if~$v$ solves (NS) then~$e^{\sigma \sqrt{-t\Delta}}v(t)$ is a smooth function for some~$\sigma>0$.  This result was extended in~\cite{cheminleray,HS,lemarie1, lemarie2} where it is proved for instance  that
$$
 \int_{\R^3 }|\xi| \bigl(\sup_ {t\leq T}
e^{     \sqrt t|\xi|} |\wh v (t,\xi)| \bigr)^2\,d\xi + \int_0^T \int_{\R^3 }|\xi|^3 \bigl(
e^{     \sqrt t|\xi|} |\wh v (t,\xi)| \bigr)^2\,d\xi dt<\infty\,,
$$
which shows that the radius of analyticity~$R(t)$ of~$v(t)$ is bounded from below by~$\sqrt t$.  Note that the above condition is equivalent to the fact  that~$e^{  \sqrt{-t\Delta}}v(t)$ belongs to~$E^\infty_T \cap E^2_T$, where~$E^q_T$ denotes the space of vector fields~$V$ such that
$$
\|V\|_{E^q_T} \eqdefa  \bigl \| 2^{j  \left(\frac12+\frac 2 q \right)}  \|\D_j V\|_{L^q([0,T]; L^2(\R^3))} \bigr\|_{\ell^2(\ZZ)}\,.
$$
This type of result is also known to hold in the more general context of~(SP) (see~\cite{ferrarititi,katomasuda} for instance) and may be stated as follows.
\begin{theo}
\label {KatoAbstraitAna}
{\sl
The solution constructed in Theorem{\rm\refer {KatoAbstrait}}  is analytic for positive~$t$ with radius of analyticity~$R(t)$ greater than~$\sqrt t$.}
\end{theo}
The purpose of this  work  is the proof of the following improved theorem.
\begin{theo}
\label {KatoAbstraitAna+}
{\sl
\begin{itemize}
\item[(a)] If~$\d$ is positive, the solution constructed in Theorem{\rm\refer {KatoAbstrait}}  satisfies
$$
\liminf _{t\rightarrow 0}  \frac {R(t) } {   t^\frac12  {\sqrt{-   \log  \bigl( t \|U_0\|^{\frac 2 \d }_{H^{s_{\rm crit} +\d} }\bigr) }}}\geq \sqrt {2 \d}\,.
$$
\item[(b)] If~$\d=0$, then, for any positive~$\e$ small enough, we have
$$
\liminf _{t\rightarrow 0}  \frac {R(t) } {     t^\frac12  {\sqrt{ -   \log  \bigl( \|e^{\e \tau\D} U_0\| _{K^p_t} }\bigr)}}\geq 2\sqrt {1-\e}\,.
$$
In particular~$\ds  \lim_{t\rightarrow 0} \frac {R(t)} { t^\frac12} =\infty$.
\end{itemize}}
\end{theo}
We remark that in the case of three-dimensional incompressible Navier-Stokes system
\begin{equation*} \rm{(NS)}
\quad \left\{\begin{array}{l}
\displaystyle \pa_t u + u\cdot\nabla u -\Delta u+\nabla p=0\, , \qquad (t,x)\in\R^+\times\R^3\, , \\
\displaystyle \dive u = 0\, , \\
\displaystyle  u|_{t=0}=u_0\, ,
\end{array}\right.
\end{equation*}
where $u=(u^1,u^2,u^3)$ denotes the velocity of the fluid and $p$ the scalar pressure function,
 part~(a) of Theorem\refer  {KatoAbstraitAna+}
 coincides with Theorem 1.3 of \cite{HS}. Moreover, the main idea used to prove Theorem \ref{KatoAbstraitAna+}
 can be applied to investigate the  radius of analyticity of any global solution of~(NS). More precisely
 we can prove the following result.
\begin{cor}\label{S1cor1}
{\sl Let $u\in C([0,\infty); H^{\frac12}(\R^3))$ be a global solution of (NS). Then one has
\beq \label{S1eq1}
\liminf_{t\to\infty}\frac{R(t)}{ t^\frac12}=\infty\, .
\eeq}
\end{cor}

\section{Proof of Theorem\refer {KatoAbstraitAna+}}
We shall perform all our computations on  the approximated system
$$
{{\rm(SP}_n\rm{)}}\quad \left\{
\begin{array}{c}
\partial_t U  -\D U  =  P_n(U)\\
U_{|t=0} =U_0\,,
\end{array}
\right. \with
P_{n,j}(U)\eqdefa \sumetage {\ell\in \N^N} {|\ell|=k } {\bf 1}_{B(0,n)} (D) A_{j,\ell} (D)( U^\ell)
$$
for~$j$ in $\{1,\dots, N\},$ and where we have written ${\bf 1}_{B(0,n)}$ for the characteristic function of the ball~$B(0,n)\eqdefa \left\{\ \xi\in \R^d; \ |\xi|\leq n \right\}.$ The system ${{\rm(SP}_n\rm{)}} $  is an ordinary differential equation  in  all Sobolev spaces. All the quantities we shall write are defined   in this case, and we neglect the index~$n$  in all that follows. We also   skip the final stage
 of passing to  the limit when~$n$  tends to infinity.

\medskip

 Let us consider three positive real numbers~$T$,~$\lam$ and~$\e$ which will be chosen later on in the proof. Motivated by \cite{HS},
we define
\beq
\label {KatoAbstraitAna+demoeq1}
U_a(t,x) \eqdefa \cF^{-1} \bigl( e^{-\frac {\lam^2} {4(1-\e)} \frac t T +\lam \frac t {\sqrt T} |\xi|} |\wh U(t,\xi)|\bigr)\,.
\eeq\,
The main point is that the function~$U_a$ behaves like a solution of a modified system~(SP) where the viscosity  is~$\e$ instead of~$1$ and the
 non-linear term  has a  factor~$e^{-\frac {\lam^2(k-1)} {4(1-\e)} }$.  We shall make this idea more precise in what follows. \smallskip

The key ingredient used to prove Theorem \ref{KatoAbstraitAna+} will be the following lemma.
\begin{lemme}\label{S2lem1}
{\sl Let $U_a$ be defined by \eqref{KatoAbstraitAna+demoeq1}. Then for any $p$ in~$ \left]\max(2/\al,k), \infty\right[,$
there exists a positive constant $C_{k,\e}$ such that
\beq
\label {KatoAbstraitAna+demoeq5}
\|U_a\|_{K^p_T}  \leq  \|e^{\e t\D} U_0\| _{K^p_T}
+C_{k,\e}\Bigl(e^{\frac {\lam^2} {4(1-\e)} } \| U_a\|_{K^p_T}\Bigr)^{k-1}  \| U_a\|_{K^p_T}\,.
\eeq}
\end{lemme}

\begin{proof} A solution of~${{\rm(SP}_n\rm{)}} $ satisfies
\beq
\label {KatoAbstraitAna+demoeq2}
|\wh U(t,\xi)| \leq e^{-t|\xi|^2}|\wh U_0(\xi)| +C \int_0^t e^{-(t-t')|\xi|^2} |\xi|^\b \bigl(\underbrace {|\wh U(t')| \star \dots \star |\wh U(t')|}_{k\ {\rm times}} \bigr)(\xi) dt'\,.
\eeq
Let us observe that
$$
-\frac {\lam^2} {4(1-\e)} \frac 1 T + \frac  \lam {\sqrt T} |\xi| -|\xi|^2 \leq -\e|\xi|^2.
$$
Thus by definition\refeq   {KatoAbstraitAna+demoeq1}, we infer from\refeq {KatoAbstraitAna+demoeq2}   that
$$
\wh U_a(t,\xi) \leq e^{-\e t|\xi|^2}|\wh U_0(\xi)| +C \int_0^t e^{-\e(t-t')|\xi|^2 }
e^{-\frac {\lam^2} {4(1-\e)} \frac {t'} T +\lam \frac {t'} {\sqrt T} |\xi|} |\xi|^\b \bigl(\underbrace {|\wh U(t')| \star \dots \star |\wh U(t')|}_{k\ {\rm times}} \bigr)(\xi) dt'.
$$
Notice that
\beno
 \bigl(\underbrace {|\wh U(t')| \star \dots \star |\wh U(t')|}_{k\ {\rm times}} \bigr)(\xi)=
\int_{\sum_{\ell=1} ^k\xi_\ell =\xi} \Bigl(\prod_{\ell=1}^k|\wh U(t',\xi_\ell)| \Bigr) d\xi_1\dots d\xi_k\,,
\eeno
and using that, for any~$(\xi_j)_{1\leq j\leq k}$ in~$(\R^d)^k$ such that~$\ds \sum_{j=1}^k \xi_j =\xi$, there holds~$\ds
e^{|\xi|} \leq \prod_{j=1}^k e^{|\xi_j|}, $
we infer~that
\beq
\label {KatoAbstraitAna+demoeq3}
\begin{aligned}
& \wh U_a(t,\xi) \leq e^{-\e t|\xi|^2}|\wh U_0(\xi)| +Ce^{\frac {\lam^2(k-1)} {4(1-\e)} } 
\int_0^t e^{-\e(t-t')|\xi|^2 }  |\xi|^\b \bigl(\underbrace {\wh U_a(t') \star \dots \star \wh U_a(t')}_{k\ {\rm times}} \bigr)(\xi) dt'\,.
\end{aligned}
\eeq
Let  us recall the following result on products in Sobolev spaces: for any positive real number~$s$,   smaller than~$d/2$ and greater than~$  d/2-  d/k$, there holds
\beq\label{lawp}
 \Bigl\|\prod_{\ell=1} ^k a_k\Bigr\|_{ H^{ks-(k-1) \frac d 2} } \leq   C_k \prod_{\ell=1} ^k \|a_k\|_{H^s}\, .
 \eeq
Now let us choose~$p$  in~$ [1,\infty]$ such that
\beq
\label {KatoAbstraitAna+demoeq4}
0<  \frac 2p<\alpha
\quad\hbox{and set}\quad
s_p \eqdefa  s_{\rm crit} + \frac2p = \frac d 2 +\frac 2 p -\al\, .
\eeq
Notice that $$ \underbrace {\wh U_a(t') \star \dots \star \wh U_a(t')}_{k\ {\rm times}} = (2\pi)^{kd}  \cF (U_a^k(t'))\,.$$ The
 assumption that $\al\leq \frac{d}k $   implies that $\al<\frac{d}k+\frac2p$ and $s_p>\frac{d}2-\frac{d}k,$  so~\eqref{lawp} ensures that
$$
 \bigl(\underbrace {\wh U_a(t') \star \dots \star \wh U_a(t')}_{k\ {\rm times}} \bigr)(\xi) \leq C^{k}
 |\xi|^{-\left( s_p+ (k-1)\left (\frac 2 p { -}\al \right)\right)} f(t',\xi) \|U_a(t')\|_{H^{s_p} }^k\quad \hbox{with}\  \|f(t')\|_{L^2(\R^d)} =1\,.
$$
As~$\al(k-1) =2-\b$, plugging the above inequality in\refeq  {KatoAbstraitAna+demoeq3} gives
$$
 \wh U_a(t,\xi) \leq e^{-\e t|\xi|^2}|\wh U_0(\xi)| +C^ke^{\frac {\lam^2(k-1)} {4(1-\e)} } \int_0^t e^{-\e(t-t')|\xi|^2 }  |\xi|^{-s_p+2-(k-1)\frac 2 p} f(t',\xi) \| U_a(t')\|_{H^{s_p}}^k  dt'.
$$
By multiplication of this inequality by~$t^{\frac 1 p} |\xi|^{s_p}$ and  by definition of the norm~$\|\cdot\|_{K^p_T}$ in \eqref{Kato}, we get,  for any~$t$ in the interval~$[0,T]$,
$$
\longformule{
t^{\frac 1 p} |\xi|^{s_p}  \wh U_a(t,\xi) \leq t^{\frac 1 p} |\xi|^{s_p}e^{-\e t|\xi|^2}|\wh U_0(\xi)|
}
{ {}+C^ke^{\frac {\lam^2(k-1)} {4(1-\e)} } \| U_a\|_{K^p_T}^k  t^{\frac 1p} \int_0^t e^{-\e(t-t')|\xi|^2 } \frac 1  {\left(t'\right)^{\frac k p}} |\xi|^{2\left(1-\frac{k-1}p\right)} f(t',\xi)  dt'.
}
$$
If we assume that~$p$ is greater than~${ k}$,  the function~$y{\in [0,\infty]\mapsto y^{ 1-\frac {k-1} p}  e^{-\e y}}$ is bounded, we infer that for any~$t$ in the interval~$[0,T]$,
$$
\longformule
{
t^{\frac 1 p} |\xi|^{s_p}  \wh U_a(t,\xi) \leq t^{\frac 1 p} |\xi|^{s_p}e^{-\e t|\xi|^2}|\wh U_0(\xi)|
}
{ {}+C_{k,\e} e^{\frac {\lam^2(k-1)} {4(1-\e)} } \| U_a\|_{K^p_T}^k  t^{\frac 1p} \int_0^t \frac 1 {(t-t')^{1-\frac {k-1}p }}  \frac 1  {t'^{\frac k p}}  f(t',\xi)  dt'.
}
$$
Taking the~$L^2$ norm with respect to the variable~$\xi$ gives, for any~$t$ in the interval~$[0,T]$,
$$
{
t^{\frac 1 p} \|U_a(t)\|_{H^{s_p}}  \leq t^{\frac 1 p} \|e^{\e t\D} U_0\| _{H^{s_p}}
}
{ {}+C_{k,\e} e^{\frac {\lam^2(k-1)} {4(1-\e)} } \| U_a\|_{K^p_T}^k  .
}
$$
Taking the supremum with respect to~$t$ in the interval~$[0,T]$ gives \eqref{KatoAbstraitAna+demoeq5}.
\end{proof}
\smallskip

Let us now turn to the proof of Theorem  \ref{KatoAbstraitAna+}.

\begin{proof}[Proof of Theorem  {\rm\ref{KatoAbstraitAna+}}] Let $U$ and $U_a$ be determined respectively  by~${{\rm(SP}_n\rm{)}} $ and \eqref{KatoAbstraitAna+demoeq1}.
We make the following induction hypothesis
\beq
\label {KatoAbstraitAna+demoeq6}
 \| U_a\|_{K^p_T} \leq  c_{k,\e}  e^{-\frac {\lam^2} {4(1-\e)} } \with
 c_{k,\e} \eqdefa \frac 1 {\left(4C_{k,\e}\right)^{\frac 1{ k-1}} } \,
 \eeq with $C_{k,\e}$ being determined by Lemma \ref{S2lem1}.
As long as this  induction hypothesis is satisfied, Inequality\refeq {KatoAbstraitAna+demoeq5} becomes
\beq
\label {KatoAbstraitAna+demoeq7}
\|U_a\|_{K^p_T}  \leq   {\frac43}\|e^{\e t\D} U_0\| _{K^p_T} \,.
\eeq
Now let us distinguish the case  when~$U_0$  belongs to the space~$H^{{ s_{\rm crit}}+\d}$ from the case when~$U_0$ belongs only to  the critical space~$H^{s_{\rm crit}}$.

\begin{itemize}
\item [(a)] The case  when~$U_0$  belongs to the space~$H^{{ s_{\rm crit}}+\d}.$
\end{itemize}
We first observe that \beq
\label {KatoAbstraitAna+demoeq8}
\|e^{\e t\D} U_0\| _{K^p_T} \leq C_{\d,p} T^{\frac \d 2} \|U_0\|_{H^{s_{\rm crit} +\d}}
\eeq
Let us define
$$
T_{\e} (U_0) \eqdefa \eta_\e \|U_0\|_{H^{s_{\rm crit} +\d}}^{-\frac 2 \d}\with \eta_\e \eqdefa \biggl( \frac {c_{k,\e}}{{ 2C_{\d,p}}}\biggr)^{ \frac 2 \d}\cdotp
$$
By definition of~$T_{\e} (U_0)$, we have that for any~$T\leq T_{\e} (U_0)$,
$$
  2 C_{\d,p} T^{\frac \d 2} \|U_0\|_{H^{s_{\rm crit} +\d}} \leq c_{k,\e}.
$$
Now let us define
$$
\lam_T \eqdefa   {\sqrt {2 \d(1-\e) } \log^\frac12 \biggl(\frac {\eta_\e  }  { T \|U_0\|_{H^{s_{\rm crit} +\d}}^{{\frac2{\d}}}}\biggr)} \,\cdotp
$$
Then for $T\leq T_\e(U_0),$  we deduce from the Ineqalities\refeq  {KatoAbstraitAna+demoeq7} and\refeq {KatoAbstraitAna+demoeq8}
that
$$
\|U_a\|_{K^p_T}\leq \frac43C_{\d,p} T^{\frac \d 2} \|U_0\|_{H^{s_{\rm crit} +\d}}<2C_{\d,p} T^{\frac \d 2} \|U_0\|_{H^{s_{\rm crit} +\d}}
= c_{k,\e} e^{-\frac {\lam_T^2} {4(1-\e)} }.
$$
This in turn shows that \eqref{KatoAbstraitAna+demoeq6} holds for $T\leq T_\e(U_0).$ Furthermore, according to
\eqref{Kato} and \eqref{KatoAbstraitAna+demoeq1}, there holds
$$
T^{\frac 1 p}   \bigl\|  e^{\lam_T \sqrt T {|D|}} U(T)\bigr\|_{H^{s_{\rm crit} +\frac 2 p}} \leq c_{k,\e}.
$$
As~$\d$ is less than~$\al$, taking~$\ds p =\frac 2 \d$ ensures that
$$
\forall T \leq T_\e(U_0)\,,\ R(T) \geq    { \sqrt {2 \d(1-\e) } T^\frac12 \log^\frac12 \biggl(\frac {\eta_\e  }  { T \|U_0\|^{\frac 2\d }_{H^{s_{\rm crit} +\d}}}\biggr)}\,\cdotp
$$
This inequality means exactly that
$$
\liminf _{T\rightarrow 0}  \frac {R(T) } {   T ^\frac12  {\sqrt{-  \log   \bigl( T \|U_0\|^{\frac 2 \d }_{H^{s_{\rm crit} +\d} }\bigr) }}} \geq \sqrt {2 \d(1-\e) } \,.
$$
Due the  fact that~$\e$ is arbitrary, we conclude the proof of  part (a) of Theorem  \ref{KatoAbstraitAna+}.

\begin{itemize}
\item [(b)] The case  when~$U_0$  belongs to the critical space~$H^{{ s_{\rm crit}}}.$
\end{itemize}
 Let us use the fact that in this case
\beq\label{S2eq8}
\lim_{T\rightarrow 0} \|e^{\e t\D} U_0\| _{K^p_T}=0\,.
\eeq
Then we consider~$T_\e(U_0)$ such that
\beq \label{eq2.10}
\|e^{\e t\D} U_0\| _{K^p_{T_\e(U_0)}} \leq  c_{k,\e} \,.
\eeq
For any~$T\leq  T_\e(U_0)$, let us define
$$
\lam_T \eqdefa  2  {(1-\e)^\frac12 \log^\frac12 \biggl( \frac {c_{k,\e}} {2\|e^{\e t\D} U_0\| _{K^p_T} }\biggr) }
\,.
$$
Then it follows from Inequality\refeq  {KatoAbstraitAna+demoeq7}  that
$$
\|U_a\|_{K^p_T}\leq \frac43\|e^{\e t\D} U_0\| _{K^p_T}<2\|e^{\e t\D} U_0\| _{K^p_T}
= c_{k,\e} e^{-\frac {\lam_T^2} {4(1-\e)} }.
$$
This shows that \eqref{KatoAbstraitAna+demoeq6} indeed  holds for $T\leq T_\e(U_0).$ Furthermore, according to
\eqref{Kato}  and \eqref{KatoAbstraitAna+demoeq1}, there holds
$$
T^{\frac 1 p}   \bigl\|  e^{\lam_T \sqrt T {|D|}} U(T)\bigr\|_{H^{s_{\rm crit} +\frac 2 p}} \leq c_{k,\e}\,.
$$
By definition of~$\lam_T$ this means in particular that
$$
\forall T\leq  T_\e(U_0)\,,\ R(T) \geq 2  {(1-\e)^\frac12  T^\frac12\log^\frac12 \biggl( \frac {c_{k,\e}} {2\|e^{\e t\D} U_0\| _{K^p_T} }\biggr)}
\,.
$$
This inequality means exactly that for any small  strictly positive~$\e$, we have
$$
\liminf _{T\rightarrow 0}  \frac {R(T) } {    {  T^\frac12 \sqrt{- \log^\frac12 \bigl( \|e^{\e t\D} U_0\| _{K^p_T} }\bigr)}}\geq \lim_{T\rightarrow 0}2  {(1-\e)^\frac12\biggl(1- \frac {\log c_{k,\e}} {\log\bigl(2\|e^{\e t\D} U_0\| _{K^p_T}\bigr) }\biggr)^\frac12}\,,
$$  which together with \eqref{S2eq8} ensures part (b) of of Theorem  \ref{KatoAbstraitAna+}.
This completes the proof of the theorem.
\end{proof}
 \section{Proof of Corollary \ref{S1cor1}}
 Let $u\in C([0,\infty); H^{\frac12}(\R^3))$ be a global solution of the Navier-Stokes
system (NS) with initial data $u_0.$ Then it follows from \cite{chemin99} that this solution is unique, so that  applying
Theorem 2.1 of \cite{GIP03} yields
\beq\label{S2eq11}
\lim_{t\to\infty} \|u(t)\|_{H^{\frac12}}=0\, .
\eeq
Moreover, for any $t_0>0,$ $u$ verifies
\begin{equation*} (NS_{t_0})
\quad \left\{\begin{array}{l}
\displaystyle \pa_t u + u\cdot\nabla u -\Delta u+\nabla p=0, \qquad (t,x)\in ]t_0,\infty[\times\R^3, \\
\displaystyle \dive u = 0, \\
\displaystyle  u|_{t=t_0}=u(t_0).
\end{array}\right.
\end{equation*}
Similar to \eqref{KatoAbstraitAna+demoeq1} we denote
\beq
\label{S2eq12}
u_{a,t_0}(t,x) \eqdefa \cF^{-1} \bigl( e^{-\frac {\lam^2} {4(1-\e)} \frac {t-t_0} T +\lam \frac {t-t_0} {\sqrt T} |\xi|} |\wh u(t,\xi)|\bigr),
\eeq
and
\beno
\|u\|_{K^p_{t_0,T}}\eqdefa \sup_{t\in [t_0,t_0+T]}\Bigl((t-t_0)^{\frac1p}\|u(t)\|_{H^{\frac12+\frac2p}}\Bigr).
\eeno
Then along the same line to proof of Lemma \ref{S2lem1}, we deduce that
\beq \label{S2eq13}
\|u_{a,t_0}\|_{K^p_{t_0,T}}  \leq  \|e^{\e (t-t_0)\D} u(t_0)\| _{K^p_{t_0,T}}
+C_{\e}e^{\frac {\lam^2} {4(1-\e)} } \| u_{a,t_0}\|_{K^p_{t_0,T}}^2\,.
\eeq
By  \eqref{S2eq11}, we can choose $t_0$ so large that
\beno
\|u(t_0)\|_{H^{\frac12}}\leq \frac{c_\e}{2K_\e}
\eeno
with $K_\e$ being determined by $\|\e^{\e t\D}u_0\|_{K^p_\infty}\leq K_\e\|u_0\|_{H^{\frac12}}.$

Let us make the induction assumption that
\beq\label{S2eq14}
 \| u_{a,t_0} \|_{K^p_{t_0,T}} \leq  c_{\e}  e^{-\frac {\lam^2} {4(1-\e)} } \with
 c_{\e} \eqdefa \frac 1 {4C_{\e}} \, \cdotp
 \eeq
Then as long as the induction assumption is satisfied, we infer from \eqref{S2eq13} that
\beno
\| u_{a,t_0} \|_{K^p_{t_0,T}}\leq \frac43\|e^{\e (t-t_0)\D} u(t_0)\| _{K^p_{t_0,T}}<2\|e^{\e (t-t_0)\D} u(t_0)\| _{K^p_{t_0,T}}
\leq 2K_\e\|u(t_0)\|_{H^{\frac12}}\leq c_\e.
\eeno
So that for any $T>0,$ we define
$$
\lam_T \eqdefa  2  {(1-\e)^\frac12 \log^\frac12 \biggl( \frac {c_{\e}} {2K_\e\|u(t_0)\| _{H^{\frac12} }}\biggr) }
\,,
$$
we have
\beno
\| u_{a,t_0} \|_{K^p_{t_0,T}}< 2K_\e\|u(t_0)\|_{H^{\frac12}}\leq c_\e e^{-\frac {\lam_T^2} {4(1-\e)} }.
\eeno
This in turn shows that \eqref{S2eq14} holds for any $T>0.$ \eqref{S2eq14} in particular implies that
$$
T^{\frac 1 p}   \bigl\|  e^{\lam_T \sqrt T {|D|}} u(t_0+T)\bigr\|_{H^{\frac12 +\frac 2 p}} \leq c_{\e}\,.
$$
As a result, it comes out
$$
\forall T\,,\ R(t_0+T) \geq 2  {(1-\e)^\frac12  T^\frac12\log^\frac12 \biggl( \frac {c_{\e}} {2K_\e\|u(t_0)\| _{H^{\frac12} } }\biggr)}
\,,
$$
from which  we infer that
$$
\liminf _{T\rightarrow \infty}  \frac {R(t_0+T) } { \sqrt{t_0+T}}=\liminf _{T\rightarrow \infty}  \frac {R(t_0+T) } { \sqrt{T}}
\geq 2  {(1-\e)^\frac12  \log^\frac12 \biggl( \frac {c_{\e}} {2K_\e\|u(t_0)\| _{H^{\frac12} } }\biggr)}\,.
$$
This together with \eqref{S2eq11} ensures \eqref{S1eq1}. This finishes the proof of Corollary \ref{S1cor1}.
\qed

 \bigbreak \noindent {\bf Acknowledgments.} Part of the work was done when P. Zhang was visiting
Laboratoire J. L. Lions of  Sorbonne Universit\'e in the fall of 2018. He would like to thank the hospitality of the Laboratory.
P. Zhang is partially supported
by NSF of China under Grants   11371347 and 11688101, and innovation grant from National Center for
Mathematics and Interdisciplinary Sciences.

   \end{document}